\documentclass[10pt]{amsart}

\usepackage[utf8]{inputenc}
\usepackage[T1]{fontenc}
\usepackage[english]{babel}
\usepackage[autostyle=true]{csquotes}
\usepackage{latexsym, amssymb, amsmath}
\usepackage{nameref,hyperref,cleveref}
\usepackage{amsthm}
\usepackage{enumerate}

\usepackage{url}
\usepackage{float}
\usepackage{caption}
\usepackage{subcaption}
\usepackage{tikz-cd}
\usepackage{color}  
\usepackage{hyperref}
\usepackage{booktabs}
\usepackage{array}
\usepackage[symbol]{footmisc}
\hypersetup{%
colorlinks=true,
linkcolor=black,
filecolor=magenta,      
urlcolor=cyan,
linktoc=all,    
}
\usepackage[ruled,vlined]{algorithm2e}
\usepackage{graphicx}
\graphicspath{{./images/}}
\usepackage{amscd} 
\usepackage{tikz}
\usepackage{pgfplots}
\pgfplotsset{compat=1.15}
\usepackage{mathrsfs}
\usetikzlibrary{arrows}

\newtheorem{theorem}{Theorem}[section]
\newtheorem{lemma}[theorem]{Lemma}
\newtheorem{corollary}[theorem]{Corollary}
\newtheorem{proposition}[theorem]{Proposition}

\theoremstyle{definition}
\newtheorem{definition}[theorem]{Definition}
\newtheorem{remark}{Remark}[section]
\newtheorem{example}{Example}[theorem]
\newtheorem{problem}[theorem]{Problem}

\DeclareMathOperator{\STR}{STR}
\DeclareMathOperator{\rSTR}{rSTR}

\usepackage[%
backend=bibtex,%
style=numeric-comp,%
sorting=nyt,%
hyperref,%
maxcitenames=4, mincitenames=4,
maxbibnames=99, minbibnames=99 %
]{biblatex}

\AtBeginBibliography{}

\begin{document}

\title[An extension formula for right Bol loops]{An extension formula for right Bol loops arising from Bol reflections}

\author{Mario Galici}
\address{Dipartimento di Matematica e Informatica, Università degli Studi di Palermo,
Italy}
\email{mario.galici@unipa.it}

\author{Gábor P.\ Nagy}
\address{Bolyai Institute, University of Szeged, 6725 Szeged, Aradi v\'ertan\'uk tere 1, Hungary and HUN-REN-ELTE Geometric and Algebraic Combinatorics Research Group, Eotv\"os Lor\'and University, 1117 Budapest, P\'azm\'any s. 1/c, Hungary}
\email{nagyg@math.u-szeged.hu}

\begin{abstract}
We study a new extension formula for right Bol loops. We prove the necessary or sufficient conditions for the extension to be right Bol. We describe the most important invariants: right multiplication group, nuclei, and center. We show that the core is an involutory quandle which is the disjoint union of two isomorphic involutory quandles. We also derive further results on the structure group of the core of the extension.
\end{abstract}



\maketitle

\section{Introduction}

A \emph{quasigroup} is a set $L$ endowed with a binary operation $x\cdot y$ such that equations $a\cdot x=b$, $y\cdot a=b$ have unique solutions for $x,y$. Solutions are denoted by divisions on the left and the right $x=a\backslash b$, $y=b/a$. \textit{Loops} are quasigroups with a unit element $1$. The multiplication sign is often ignored; $(x\cdot y)\cdot z$ is written as $xy\cdot z$. The \emph{left} and \emph{right multiplication maps} $L_a:x\mapsto ax$, $R_a:x\mapsto xa$ are invertible maps of $L$ to itself. A loop is said to be a \emph{right Bol} loop if it satisfies the following identity for all $x,y,z\in L$:
\begin{equation}\label{eq:rightbol}
((xy)z)y=x((yz)y).
\end{equation}
Equivalently, $R_yR_zR_y=R_{(yz) y}$ for every $y,z\in L$. In a right Bol loop, the left and right inverses $x\backslash 1$ and $1/x$ of any element $x$ coincide; we call it the \emph{inverse} of $x$, denoted by $x^{-1}$. Furthermore, any right Bol loop is power associative, meaning that every element generates a cyclic group. Furthermore, $R_x^{-1}=R_{x^{-1}}$, that is, $(xy)y^{-1}=x$ for all $x,y$. If a loop satisfies the right Bol property \eqref{eq:rightbol} and its opposite $x(y(xz)) = (x(yx))z$, then it is called a \emph{Moufang loop.} Moufang loops are \emph{diassociative,} that is, any two elements generate an associative subloop. In particular, the inverse map is an antiautomorphism: $(xy)^{-1}=y^{-1}x^{-1}$. However, the so-called \textit{automorph inverse property}
\begin{align} \label{eq:AIP}
(xy)^{-1}=x^{-1}y^{-1}
\end{align}
plays an important role for non-Moufang right Bol loops. Loops satisfying \eqref{eq:AIP} are called \textit{AIP} loops. Finally, a loop is said to be \emph{right conjugacy closed} if $R_xR_yR_x^{-1}=R_{xy/x}$ holds for all $x,y,z$. Note that if a right Bol loop has central squares then it is right conjugacy closed, see \cite[Theorem 1.4.4]{nagy_strambach_1994}.

The \textit{nuclei} of a loop are defined as follows:
\begin{align*}
N_\lambda &= \{ n\in L \mid (nx)y=n(xy) \;\text{for all $x,y\in L$}\} && \text{(left nucleus)} \\
N_\mu &= \{ n\in L \mid (xn)y=x(ny) \;\text{for all $x,y\in L$}\} && \text{(middle nucleus)} \\
N_\rho &= \{ n\in L \mid (xy)n=x(yn) \;\text{for all $x,y\in L$}\} && \text{(right nucleus)}
\end{align*}
The \textit{commutant} of a loop $L$ is 
\[C(L) = \{z\in L \mid xz=zx \text{ for all $x\in L$}\}.\]
The \textit{center} of $L$ is the intersection of the commutant and the nuclei:
\[Z(L) = C(L) \cap N_\lambda \cap N_\mu \cap N_\rho. \]
The center is always a normal subloop in $L$. In right Bol loops, the right and middle nuclei coincide, and in Moufang loops, they coincide with the left nucleus as well. In groups, the commutant is the same as the center, and hence normal. In Moufang loops, the commutant is a subloop, but not necessarily normal (\cite{GrishkovZavarnitsine}). There are infinite classes of right Bol loops, in which the commutant is not even a subloop (\cite{CommutantBolLoop}).

This paper deals with a new extension formula for loops. Let $(L,\cdot)$ be a loop, and let 
\begin{align*}
\mathcal{T}&=\{t_a \mid a \in L\}, \\
\mathcal{V}&=\{v_a \mid a \in L\}, 
\end{align*}
be two disjoint copies of $L$. Define a product on $\widetilde{L}=\mathcal{T} \cup \mathcal{V}$ by
\begin{equation} \label{eq:extformula}
t_a\cdot t_b= t_{ab}, \quad  t_a \cdot v_b = v_{ab}, \quad v_a \cdot t_b = v_{ab^{-1}}, \quad v_a \cdot v_b = t_{ab^{-1}}.
\end{equation}
This makes $(\widetilde{L},\cdot)$ into a loop with unit $t_1$, and $\mathcal{T}$ is a normal subloop of index $2$, isomorphic to $L$. We will prove:

\begin{theorem} \label{thm:Bolcond}
$\widetilde{L}$ is a right Bol loop if and only if $L$ is a right Bol loop with $x^2\in Z(L)$ for every $x\in L$. Moreover, the following are equivalent:
\begin{enumerate}[(i)]
\item $\widetilde{L}$ is Moufang.
\item $\widetilde{L}$ is associative.
\item $L$ is an abelian group. 
\end{enumerate}
\end{theorem}

The class of right Bol loops with central squares contains the class of Bol loops of exponent $2$. The latter is a very rich class, containing simple proper finite Bol loops of exponent $2$ (\cite{Aschbacher}, \cite{Nagyexp2}). Right conjugacy closed Moufang loops are also called \textit{extra loops} (\cite{Fenyves1}, \cite{Fenyves2}). In extra loops, all squares are in the nucleus. An important subclass is the class of \textit{code loops,} which are related to doubly even binary linear codes (\cite{NagyCodeLoops}, \cite{Griess}, \cite{Hsu}). Moreover, Bol loops of exponent 4 with central squares are important in our context because, in this case, the process of constructing the extension can be iterated. Among these loops, we find code loops and \emph{small Frattini Bol 2-loops} \cite{nagysmallfrattini}.

We will prove:
\begin{theorem} \label{thm:nuclei}
Let $L$ be a right Bol loop with central squares. 
\begin{enumerate}[(i)]
\item The right nucleus of the extension $\widetilde{L}$ is
\[N_\rho(\widetilde{L})=\{t_z, \ v_z \mid z\in Z(L) \}.\]
\item For the left nucleus of $\widetilde{L}$ it holds that 
\[N_\lambda(\widetilde{L})\cap \mathcal{T}=\{t_n \mid n\in N_\lambda(L) \}\cong N_\lambda(L).\]
Furthermore, if $L$ is a non-abelian group, then 
\begin{equation}
N_\lambda(\widetilde{L})=\mathcal{T}\cong L;
\end{equation}
If $L$ is an AIP loop, then
\begin{equation}
N_\lambda(\widetilde{L})=\{t_n, \ v_n \mid n\in N_\lambda(L) \}.
\end{equation}
\end{enumerate}
\end{theorem}

\bigskip

The \textit{core} of a right Bol loop is the binary operation 
\[x+y=(yx^{-1})y.\]
It satisfies the identities
\begin{align}
x+x&=x  && \text{(idempotent)} \\ 
(x+y)+y&=x  &&\text{(involutorial)}\\ 
(x+y)+z&=(x+z)+(y+z)  &&\text{(right distributive)}
\end{align}
A binary structure with these properties is also called an \textit{involutorial quandle.} Quandles are algebraic structures associated with knots: given a finite quandle and a cocycle, one can construct a knot invariant. Involutorial quandles are also called abstract symmetric spaces (in the sense of Loos \cite{Loos}).

Not all involutorial quandles happen to be the core of some right Bol loop. For example, if the core of $(L,\cdot)$ is a trivial involutorial quandle, that is $x+y=x$ for all $x,y$, then $L$ has to be an elementary abelian $2$-group. Hence, a trivial involutorial quandle whose order is not a $2$-power cannot be the core of a right Bol loop. 

\begin{problem} \label{prob:givencore}
Let $(Q,\triangleleft)$ be an involutorial quandle. Find necessary or sufficient conditions for the existence of a right Bol loop whose core is isomorphic to $(Q,\triangleleft)$. 
\end{problem}

This problem is settled for finite involutorial quandles of odd order: they are the core of a right Bol loop if and only if they are quasigroups, see \cite[(6.14) Theorem]{Kiechle}. Notice that an involutorial quandle is a quasigroup if and only if its left multiplication maps $x\mapsto a+x$ are invertible. 

Let us call the $(Q,\triangleleft)$ a \textit{RB-quandle,} if it is the core of some right Bol loop $(Q,\cdot)$. Our long-term goal is to study RB-quandles which are disjoint unions of two proper RB-subquandles. The right Bol loop extension $\widetilde{L}$ has the following relevant property.

\begin{theorem} \label{thm:core}
Let $L$ be a right Bol loop with central squares. The core of $\widetilde{L}$ decomposes to the disjoint union of two subquandles $\mathcal{T}$ and $\mathcal{V}$, both isomorphic to the core of $L$. 
\end{theorem}

The structure of the paper is as follows. In Section~\ref{sec:prelim}, we give the necessary definitions and properties. Our focus is on the geometric and group theoretical tools, which enable us to use Aschbacher's efficient Bol loop folder method to study the extension. In Section~\ref{sec:extprops}, we prove Theorem~\ref{thm:Bolcond} on the structure of the extension. In Section~\ref{sec:nuclei}, we prove Theorem~\ref{thm:nuclei} on the center and nuclei. Finally, in Section~\ref{sec:extcore}, we prove Theorem~\ref{thm:core} and derive further results on the structure group of the core of the extension. 

\bigskip

\section{Preliminaries} \label{sec:prelim}

\subsection{Loop folders}

In \cite{Aschbacher}, Aschbacher studied the correspondence between loops and certain triples of group theoretic data in order to study finite loops using techniques from finite group theory. 

Let $L$ be a loop, $K=\{R_x\mid x\in L\}$ be the set of right translations of $L$, $G=\langle K\rangle$ be the right multiplication group of $L$ and $H=G_1$ the stabilizer of the unit $1$ of $L$. The triple $\epsilon(L)=(G,H,K)$ is called the \emph{envelope} of the loop $L$. It is known that a loop $L$ with envelope $\epsilon(L)$ is a Bol loop if and only if $K$ is a twisted subgroup of $G$. 

\begin{definition}
A \emph{loop folder} is a triple $\xi=(G,H,K)$ where $G$ is a group, $H$ is a subgroup of $G$ and $K$ is a subset of $G$ containing $1$ such that $K$ is a set of coset representatives for $G/H^g=\{H^gx \mid x\in G\}$ for each $g\in G$.
\end{definition}

A morphism $\xi\to\xi'$ of loop folders $\xi=(G,H,K)$ and $\xi'=(G',H',K')$ is a group homomorphism $\pi\colon G\to G'$ such that $H^\pi\leq H'$ and $K^\pi\subseteq K'$. 

A folder is said \emph{faithful} if $G$ acts faithfully on $\{Hx\mid x\in G\}$. A \textit{loop envelope} is a loop folder $(G,H,K)$ such that $G=\langle K \rangle$. If $L$ is a loop, then $\epsilon(L)=(G,H,K)$ is a faithful loop envelope. 

Let $\xi=(G,H,K)$ be a loop folder and define a binary operation $\ast$ on $K$ by taking $a\ast b$ to be the unique element in $K$ such that $H(a\ast b)=H(ab)$. Then $\ell(\xi)=(K,\ast)$ is a loop with identity the unit $1$ of $G$. The loop $\ell(\xi)$ is called the \emph{loop of the loop folder $\xi$}. For $\pi\colon \xi \to \xi'$ a homomorphism of loop folders, define $\ell(\pi)\colon \ell(\xi) \to \ell(\xi')$ as the restriction of $\pi$ to $K$.

If $L$ is a loop, then $\ell(\epsilon(L))\cong L$. For a faithful loop envelope $\xi$, one has $\epsilon(\ell(\xi))\cong \xi$. 

Let $G$ be a group acting transitively on a set $Q$. The subset $S\subseteq G$ is \textit{sharply transitive set,} if for any $x,y\in Q$ there is a unique $s\in S$ such that $x^s=y$. The set $R(L)=\{R_a\mid a \in L\}$ of the right translations of a loop $L$ is sharply transitive on $L$ and contains the identity $\mathrm{Id}=R_1$. 

\begin{definition}
Let $S\subseteq G$ be a sharply transitive set on $Q$, $1\in S$. Fix an element $e\in Q$. Then $\xi=(G,G_e,S)$ is a loop folder. The associated loop $\ell(\xi)$ will be denoted by $\lambda(G, S,e)$. 
\end{definition}

The operation of $\lambda(G,S,e)$ can be given as follows: $x\ast y := x^{s}$, where $s\in S$ such that $e^s=y$ for every $x,y\in Q$. 

Finally, notice that $\ell(G,H,K)$ is a right Bol loop if and only if for all $a,b\in K$, $a^{-1}, aba \in K$. In this case, $(G,H,K)$ is called a \textit{Bol loop folder.}

\subsection{Moufang loops by Chein extension}

In \cite{chein}, Chein showed a general method of constructing non-associative Moufang loops as extensions of non-abelian groups by the cyclic group of order $2$. In the context of the next theorem, a set of generators is called \emph{minimal} if it contains the smallest number of elements, and not if no proper subset is a set of generators.

\begin{theorem}[Chein]\label{theoremchein}
If $L$ is a non-associative Moufang loop for which every minimal set of generators contains an element of order $2$, then there exists a non-abelian group $G$ and an element $x\in L$ of order $2$ such that each element of $L$ may be expressed in the form $gx^\alpha$, where $g\in G$, $\alpha=0,1$, and the product of two elements of $L$ is given by
\begin{equation*}
\left(g_1x^\delta\right)\left(g_2x^\epsilon\right)=\left( g_1^\nu g_2^\mu \right) x^{\delta+\epsilon}.
\end{equation*}
where $\nu=(-1)^\epsilon$ and $\mu=(-1)^{\delta+\epsilon}$.

Conversely, given any non-abelian group $G$, the loop constructed as indicated above is a non-associative Moufang loop. 
\end{theorem}

The Moufang loop of order $2n$ arising from a group $G$ of order $n$ as in Theorem~\ref{theoremchein} is denoted by $M_{2n}(G,2)$. 

We can express Chein's formula so that it can be compared to \eqref{eq:extformula}. Replacing $g$ with $t_g$ and $gx$ with $v_g$, we have, for every $g,h\in G$
\begin{equation}
t_g t_h= t_{gh}, \quad t_g v_h=v_{g^{-1}h^{-1}}, \quad v_g t_h=v_{gh^{-1}}, \quad v_g v_h= t_{g^{-1}h}. 
\end{equation}
Conversely, \eqref{eq:extformula} can also be reformulated for a direct comparison with Chein's formula. We identify $t_a$ with $a$ and denote $v_1$ with $x$. Then, every $\ell \in \widetilde{L}$ can be written as $ax^\epsilon$, with $\epsilon=0,1$, and the product becomes
\begin{equation}
(ax^\delta)(bx^\epsilon)=(ab^\mu)x^{\delta+\epsilon},
\end{equation}
with $\mu=(-1)^\delta$.

\subsection{Nets and collineations}

This section has the aim of introducing the concept of a \emph{3-net}, specifically of a 3-net associated with a loop, and showing some relations between these two items. For further information on $3$-nets and loops, interested readers can refer to \cite{pflugfelder} or \cite{NagyStrambach}. An incidence structure is a triple $(\mathcal{P},\mathcal{L}, \mathcal{I})$ with $\mathcal{P}$ and $\mathcal{L}$ are sets whose elements are called points and lines, respectively, and $\mathcal{I}$ is a subset of the product $\mathcal{P}\times\mathcal{L}$. A point $P$ and a line $\ell$ are called incident if $(P,\ell)\in \mathcal{I}$.

\begin{definition}\label{3net}
A \emph{3-net} $\mathcal{N}$ is an incidence structure $(\mathcal{P},\mathcal{L}, \in) $ in which a point $P$ and a line $\ell$ are incident if and only if $P\in\ell$, satisfying the following axioms. 
\begin{enumerate}
\item $\mathcal{P}\neq\emptyset$.

\item $\mathcal{L}$ is the union of $3$ families        
(\emph{pencils}), namely the \emph{horizontal}, \emph{vertical} and \emph{transversal} lines, such that:
\begin{enumerate}[(i)]
\item the lines from each pencil partition $\mathcal{P}$;
\item two lines of distinct pencils have a unique point in common.
\end{enumerate}

\item Every line has exactly $n=|\mathcal{P}|$ points. 
\end{enumerate}
\end{definition}

We denote the families of horizontal, vertical, and transversal lines, respectively, by $\mathcal{H}$, $\mathcal{V}$, and $\mathcal{T}$.
A \emph{collineation} of a $3$-net is a bijection on both $\mathcal{P}$ and $\mathcal{L}$ which preserves the incidence relations between points and lines. 

It is well known that any loop $L$ can be associated with a $3$-net $\mathcal{N}(L)$ with the set of points $\mathcal{P}=L\times L$, and the lines given by:
\begin{align*}
\mathcal{H}&=\big\{ h_a=\{(x,y) \mid y=a\} \mid a\in L \big\},\\
\mathcal{V}&=\big\{ v_b= \{(x,y) \mid x=b\} \mid b\in L \big\},\\
\mathcal{T}&=\big\{ t_c=\{(x,y) \mid xy=c\} \mid c\in L \big\}.
\end{align*}
Conversely, to any $3$-net corresponds an isotopy class of loops (cf. \cite{BarlottiStrambach}, p. 10 and \cite{belousovfoundations}, p. 20). 

The most common way to represent the points of the $3$-net $\mathcal{N}(L)$ is to use their usual two coordinates $(a,b)$. However, when computing with collineations, representing points by three coordinates such as $(a,b,ab)$ can be more convenient. In this way, we can keep track of all the three lines containing the given points, which are $h_a$, $v_b$, and $t_{ab}$.

\subsection{Bol reflections}
Let us now fix an arbitrary horizontal line $h_d$. For each point $P=(a,b)\in\mathcal{P}$, consider the vertical line $v_a$ and the transversal line $t_{ab}$ incident in $P$. In particular $v_a$  and $t_{ab}$ intersect $h_d$, respectively, at the points $Q=(a,d)$ and $R=(ab/d,d)$. The intersection between the transversal line through $Q$ and the vertical line through $R$ consists of the point $P'=(ab/d, u)$, with $u\in L$ depending on $a$, $b$ and $d$, and fulfilling the equation
\begin{equation}\label{equ}
(ab)/d \cdot u = ad.
\end{equation}
The following figure provides a visual representation of these points and lines.

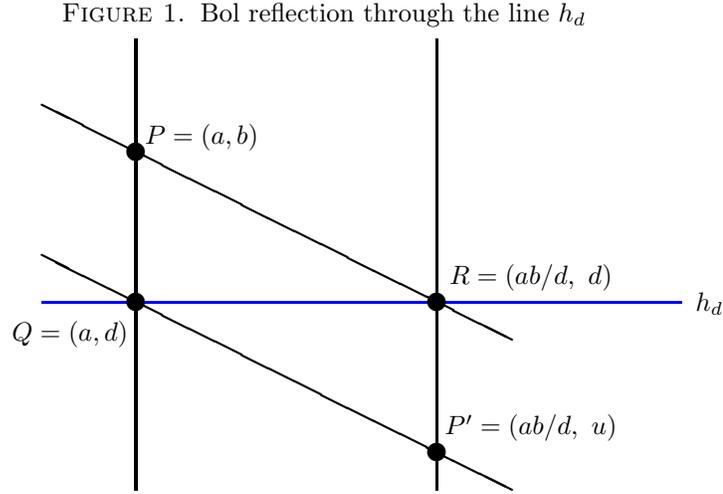
\begin{figure}[H]
\label{fig.bolreflection}
\caption{Bol reflection through the line $h_d$}
\begin{center}
\setlength{\unitlength}{0.25mm}
\begin{picture}(100,140)(0,0)
\thicklines

\put(-100,0){\textcolor{blue}{\line(1,0){340}}}
\put(-50,-100){\line(0,1){240}}
\put(-50,80){\line(2,-1){200}}
\put(-50,80){\line(-2,1){50}}
\put(-50,0){\line(2,-1){200}}
\put(-50,0){\line(-2,1){50}}
\put(110,-100){\line(0,1){240}}

\put(-50,0){\circle*{10}}
\put(-50,80){\circle*{10}}
\put(110,0){\circle*{10}}
\put(110,-80){\circle*{10}}

\put(-50,0){\makebox(-70,-10)[t]{$Q=(a,d)$}}
\put(-50,80){\makebox(70,15)[t]{$P=(a,b)$}}
\put(110,0){\makebox(100,20)[t]{$R=(ab/d,\ d)$}}
\put(110,-80){\makebox(100,20)[t]{$P'=(ab/d,\ u)$}}
\put(240,0){\makebox(30,5)[t]{$h_d$}}
\end{picture}
\end{center}
\end{figure}
\vspace{2.5cm}

For any $d\in L$, the map 
\begin{equation*}
\sigma_d\colon (a,b)\mapsto \big((ab)/d, \ u(a,b,d)\big)
\end{equation*}
is called the \emph{Bol reflection} through the line $h_d$. A Bol reflection is a collineation of the $3$-net $\mathcal{N}(L)$ if, and only if, the corresponding element $u$ in the image of a point $P=(a,b)$ does not depend on the first coordinate of $P$. This means that the whole line $h_b$ is mapped onto the line $h_u$.

It is easy to prove that the Bol reflections $\sigma_d$ are collineations of the $3$-net $\mathcal{N}(L)$ exactly when $L$ is a right Bol loop.
In this case, for any $d\in L$, the Bol reflection $\sigma_d$ is given by the map
\begin{equation}\label{bolrefl}
\sigma_d\colon (x,y,xy) \longmapsto \left( xy\cdot d^{-1}, \ dy^{-1} \cdot d, \ xd  \right).
\end{equation}	
Since in a right Bol loop $L$ the property $(ab\cdot a)^{-1}=a^{-1}b^{-1}\cdot a^{-1} $ holds for every $a,b\in L$,
we have that every Bol reflection $\sigma_d$ has order $2$. 
In particular, the Bol reflection through the horizontal line containing the unit of $L$ is 
\begin{equation}\label{bolrefl1}
\sigma_1\colon (x,y,xy) \longmapsto \left( xy,\  y^{-1} , \ x \right).
\end{equation}	
Its action consists of swapping the first and third coordinates and inverting the second. 
If we compute the composition $\sigma_1\sigma_d$ we obtain the following map
\begin{equation}\label{sigma1sigmad}
\sigma_1\sigma_d\colon (x,y,xy) \longmapsto \left( xd^{-1}, \ dy\cdot d, \ xy\cdot d  \right),
\end{equation}
which coincides with the autotopism $(R_d^{-1}, \ L_dR_d, \ R_d)$ of the right Bol loop $L$. It is easy to check that for every $d\in L$, $\sigma_d\sigma_1=\sigma_1\sigma_{d^{-1}}$.

Let us consider the sets $\Sigma:=\{\sigma_d, \ \sigma_1\sigma_d \mid d\in L\}$ and  $\Sigma_0:=\{\sigma_d \mid d\in L\}$. $\Sigma_0$ is invariant under conjugation by all collineations that preserves $\mathcal{H}$. For proper Bol loops, these are all the collineations. In general, we can say that $\Sigma_0$ is invariant under the \emph{direction preserving collineations}.
The group $\Gamma$ generated by $\Sigma_0$ is a subgroup of the full group of the collineations, and its subgroup $G:=\langle \sigma_1\sigma_d \mid d\in L \rangle$ is a subgroup of the full group of autotopisms of $L$. The reader is referred to \cite{NagyGaborCollineation} or \cite{FunkNagy} for insights into the collineation group.

\bigskip

\section{Algebraic properties of the extension} \label{sec:extprops}

In this section, $L$ is a right Bol loop, with associated $3$-net $\mathcal{N}$. The vertical and transversal lines are $v_b:x=b$ and $t_c:xy=c$. The sets of transversal and vertical lines of $\mathcal{N}$ are $\mathcal{T}$, $\mathcal{V}$. The Bol reflection on the horizontal line $y=a$ is denoted by $\sigma_a$. 

\begin{lemma}\label{lm:sharp}
The subset
\begin{align*}
\Sigma=\{\sigma_d, \ \sigma_1\sigma_d \mid d\in L\}
\end{align*}
of the group $\Gamma=\langle \sigma_d \mid d\in L \rangle$ is a sharply transitive set on $\mathcal{T}\cup \mathcal{V}$. 
\end{lemma}

\begin{proof}
Within the 3-net $\mathcal{N}(L)$, for every vertical line $v\in\mathcal{V}$ and transversal line $t\in\mathcal{T}$, there exists a unique horizontal line $h\in\mathcal{H}$, such that the associated Bol reflection interchanges $v$ and $t$. In fact, let $v_a\in\mathcal{V}$ and $t_c\in\mathcal{T}$. The intersection $v_a\cap t_c$ is the unique point $(a,b)$, where $b=a^{-1}c$. For every $y \in L$,
\begin{align*}
(a,\ y,\ ay)^{\sigma_b}=(ay\cdot b^{-1}, \ by^{-1}\cdot b, ab) = (ay\cdot b^{-1}, \ by^{-1}\cdot b, \ c)\in t_c,
\end{align*}
and 
\begin{align*}
(cy^{-1},\ y,\ c)^{\sigma_b}=(cb^{-1}, by^{-1}\cdot b, cy^{-1}\cdot b) = (a, by^{-1}\cdot b, cy^{-1}\cdot b) \in v_a.
\end{align*}
Hence, the Bol reflection $\sigma_b$ through the horizontal line $h_b$ interchanges the points of $v_a$ with the points of $t_c$.

Consequently, any composition $\sigma_1\sigma_b$ induces a permutation on the vertical lines, as well as on the transversal lines. In fact, If we consider now two vertical lines, $v_a$ and $v_c$, the autotopism $\sigma_1\sigma_b=(R_{b^{-1}}, L_bR_b, R_b)$ of $L$, with $b=c^{-1}a$, maps $v_a$ into $v_c$. Similarly, the autotopism $\sigma_1\sigma_b$ of $L$, with $b=a^{-1}c$, maps the transversal line $t_a$ into $t_b$.
\end{proof}

\begin{lemma} \label{lm:faithful}
Let $\widetilde{L}=(\mathcal{T}\cup\mathcal{V},\cdot)$, with product defined in \eqref{eq:extformula}. Then $\widetilde{L}=\lambda(\Gamma, \mathcal{T}\cup\mathcal{V}, t_1)$.
\end{lemma}
\begin{proof}
The loop $\lambda(\Gamma, \mathcal{T}\cup\mathcal{V}, t_1)$ is well-defined by Lemma~\ref{lm:sharp}, its unit element is $t_1$. The unique element of $\Sigma$ mapping $t_1$ to $t_d$ and $v_d$ is $\sigma_1\sigma_d$ and $\sigma_{d^{-1}}$, respectively. Hence, the loop operation is
\begin{align*}
t_a \cdot t_b &= t_a^{\sigma_1\sigma_b} = t_{ab},\\
t_a \cdot v_b &= t_a^{\sigma_{b^{-1}}} = v_{ab},\\
v_a \cdot t_b &= v_a^{\sigma_1\sigma_b} = v_{ab^{-1}},\\
v_a \cdot v_b &= v_a^{\sigma_{b^{-1}}} = t_{ab^{-1}}. \qedhere
\end{align*}
\end{proof}

\begin{lemma}
The loop $\widetilde{L}$ has size $|\widetilde{L}|=2|L|$. It contains $\mathcal{T}$ as a normal subloop of index $2$, which is isomorphic to $L$. Every element $v_a$ (a vertical line) has order two.
\end{lemma}
\begin{proof}
These properties follow from the product formulas. 
\end{proof}

Any equation which holds identically in $\widetilde{L}$, does hold in $L$ as well. For example, if $\widetilde{L}$ is right Bol, then $L$ is right Bol too. The converse is not true in general.

\begin{proposition}\label{prop:extbol}
$\widetilde{L}$ is a right Bol loop if, and only if, $L$ is a right conjugacy closed right Bol loop with the property 
\begin{equation}\label{eq:condition2}
ab\cdot a^{-1}=a^{-1}b\cdot a, \quad \forall \ a,b\in L.
\end{equation}
\end{proposition}

\begin{proof}
It is known that a loop is a right Bol loop if, and only if, the set of its right translations is a twisted subgroup of the group of right multiplications. In our case, since $\Sigma$ is the set of right translations, we have that $\widetilde{L}$ is a right Bol loop if, and only if, for any $a,b\in L$
\begin{equation}
\big\{ \sigma_a \sigma_b \sigma_a, \ \sigma_1\sigma_a \sigma_1\sigma_b \sigma_1\sigma_a, \ \sigma_1\sigma_a  \sigma_b  \sigma_1\sigma_a, \ \sigma_a \sigma_1\sigma_b  \sigma_a \big\} \subset \Sigma.
\end{equation}
Let us compute $(x,y,xy)^{\sigma_a\sigma_b\sigma_a}$:
\begin{align*}
(x,y,xy)^{\sigma_a\sigma_b\sigma_a}&=\Big(xy\cdot a^{-1}, \ ay^{-1}\cdot a, \ xa\Big)^{\sigma_b\sigma_a}\\
&=\Big(xa\cdot b^{-1}, \ b\left( a^{-1}y\cdot a^{-1} \right)\cdot b, \ \left( xy\cdot a^{-1} \right)b \Big)^{\sigma_a}\\
&=\Big(\left( xy\cdot a^{-1}\right)b\cdot a^{-1}, \ a\left( b^{-1} \left( ay^{-1}\cdot a\right)\cdot b^{-1}\right)\cdot a, \ \left(xa\cdot b^{-1}\right)a \Big)\\
&=\Big( \left(xy\right)\left(ab^{-1}\cdot a\right)^{-1}, \ \left(ab^{-1}\cdot a\right)y^{-1} \cdot \left(ab^{-1}\cdot a\right), \ x\left(ab^{-1}\cdot a\right)\Big).
\end{align*}
Hence we have that
\begin{equation*}
\sigma_a\sigma_b\sigma_a=\sigma_{ab^{-1}\cdot a}\in \Sigma.
\end{equation*}
Furthermore,
\begin{equation*}
\sigma_1\sigma_a\sigma_1\sigma_b\sigma_1\sigma_a=\sigma_1\sigma_a\sigma_{b^{-1}}\sigma_a=\sigma_1\sigma_{ab\cdot a}\in\Sigma.
\end{equation*}

Now, let us compute $(x,y,xy)^{\sigma_1\sigma_a\sigma_b\sigma_1\sigma_a}$:
\begin{align*}
(x,y,xy)^{\sigma_1\sigma_a\sigma_b\sigma_1\sigma_a}&=\Big(xa^{-1}, \ ay\cdot a,\  xy\cdot a\Big)^{\sigma_b\sigma_1\sigma_a}\\
&=\Big( \left(xy\cdot a\right)b^{-1} , \ b\left( a^{-1} y^{-1} \cdot a^{-1} \right) \cdot b,\ xd^{-1} \cdot f \Big)^{\sigma_1\sigma_a}\\
&=\Big( \left(xy\cdot a\right)b^{-1} \cdot a^{-1} , \ a\left(b\left( a^{-1} y^{-1} \cdot a^{-1} \right) \cdot b \right) \cdot a, \ \left(xa^{-1} \cdot b\right)a \Big).
\end{align*}
The composition  $\sigma_1\sigma_a\sigma_b\sigma_1\sigma_a$ belongs to $\Sigma$ exactly when it is equal to $\sigma_\varepsilon$, for a suitable $\varepsilon\in L$, since the first component depends on the product $xy$. By setting $x=1$, we can deduce from the third component that $\varepsilon$ must be equal to $a^{-1}b\cdot a$. Moreover, $L$ needs to satisfy the condition 
\begin{equation*}
R_a^{-1}R_bR_a=R_{a^{-1}b\cdot a}, \quad \forall \ a,b\in L,
\end{equation*}
which means that $L$ must be right conjugacy closed. Furthermore, since the first component must be equal to $xy\cdot \varepsilon^{-1}$, we find that 
\begin{equation}\label{condinv}
(a^{-1}b\cdot a)^{-1}=ab^{-1}\cdot a^{-1}  \quad \forall \ a,b\in L.
\end{equation}
Since a right conjugacy closed loop satisfies the property that for any $a,b$, $(a^{-1}b\cdot a)^{-1}=a^{-1}b^{-1}\cdot a$, we can rewrite the condition \eqref{condinv} as
\begin{equation}\label{condinv2}
ab^{-1}	\cdot a ^{-1}= a^{-1}b^{-1}\cdot b \quad \forall \ a,b\in L.
\end{equation}
Lastly, the second component of $(x,y,xy)^{\sigma_1\sigma_a\sigma_b\sigma_1\sigma_a}$ must be $\varepsilon y^{-1}\cdot \varepsilon$, and by an easy computation it is possible to see that this is equivalent to \eqref{condinv2}. 
Hence, with these properties, for all $a,b\in L$ we have that
\begin{equation*}
\sigma_1\sigma_a\sigma_b\sigma_1\sigma_a=\sigma_{a^{-1}b\cdot a}\in\Sigma,
\end{equation*}
and also 
\begin{equation}
\sigma_a\sigma_1\sigma_b\sigma_a=\sigma_1\sigma_1\sigma_a\sigma_{b}^{-1}\sigma_1\sigma_a=\sigma_1\sigma_{a^{-1}b^{-1}\cdot a}\in\Sigma. \qedhere
\end{equation}
\end{proof}

We can now prove the first part of our first main result, which replaces the equation \eqref{eq:condition2} by the condition $x^2\in Z(L)$ for all $x$. 

\begin{proof}[Proof of the first assertion of Theorem~\ref{thm:Bolcond}]
Let $L$ be a right conjugacy closed right Bol loop. Assume that \eqref{eq:condition2} holds. Then
\begin{equation*}
ab= (bb^{-1})a\cdot b= b(b^{-1}a\cdot b)=b(ba\cdot b^{-1})=b^2a\cdot b^{-1}.
\end{equation*}
Hence, for every $a,b\in L$ we have $ab=b^2a\cdot b^{-1}$, which is equivalent to $ab^2=b^2a$. 
Conversely, if $ab^2=b^2 a$ for every $a,b\in L$, then we have 
\begin{equation*}
ba\cdot b^{-1}=ba\cdot bb^{-2}=(ba\cdot b)b^{-2}=b^{-2}(ba\cdot b)=(b^{-2}b\cdot a ) b= b^{-1}a\cdot b.
\end{equation*}
Hence, since in every right conjugacy closed right Bol loops every square is in the nucleus $N(L)$, \eqref{eq:condition2} is equivalent with the fact that $x^2\in Z(L)$ for all $x$. Proposition~\ref{prop:extbol} implies the first assertion of Theorem~\ref{thm:Bolcond}.
\end{proof}

All of the six non-associative right Bol loops of order $8$ have central squares. In conclusion, if $L$ is any right Bol loop of order $8$, the corresponding $\widetilde{L}$ is a right Bol loop of order $16$. None of the proper Bol loops of order $12$ or $15$ have all squares in the center. As a result, it follows that the corresponding loop $\widetilde{L}$ cannot be a right Bol loop in either of these cases.

The situation changes for the Bol loops of order $16$. Of the $2038$ proper Bol loops $L$ with this size, using the GAP \cite{GAP4} package LOOPS \cite{loops3.4.3}, we found that a significant majority of them, precisely $1940$, satisfy the conditions of Theorem~\ref{thm:Bolcond}. Consequently, the corresponding $\widetilde{L}$ in these cases are in fact Bol loops.

\begin{remark}\label{group}
In this scenario, if $L$ is an abelian group, then $\widetilde{L}$ is the semidirect product $L\rtimes C_2$, where $L\cong \mathcal{T}$, $C_2\cong \{t_1,v_1\}$ and the action on $L$ is $x\mapsto x^{-1}$. This allows us to say that if $L$ is an abelian group, then $\widetilde{L}$ is associative but not necessarily abelian. $\widetilde{L}$ is an abelian group exactly when $L$ is an elementary abelian $2$-group.

While, if $L$ is a right Bol loop of exponent $2$, it is clear from the definition of the operation that $\widetilde{L}$ is simply the direct product $L\times C_2$.
\end{remark}

\begin{proposition}\label{propmoufang} Let $L$ be a right Bol loop with central squares. The following are equivalent:
\begin{enumerate}[(i)]
\item $\widetilde{L}$ is Moufang. \label{propcondmoufang}
\item $\widetilde{L}$ is associative.\label{propcondassoc}
\item $L$ is an abelian group. \label{propcondabel}
\end{enumerate}
\end{proposition}

\begin{proof}
Of course, for $\widetilde{L}$ to be Moufang, we also need $L$ to be Moufang. Since $L$ is right conjugacy closed, it is an extra loop.

As $\widetilde{L}$ is a right Bol loop, it is Moufang if and only if
\begin{equation}\label{condinverse}
(\ell_1\ell_2)^{-1}=\ell_2^{-1}\ell_1^{-1}
\end{equation}
for every $\ell_1,\ell_2\in \widetilde{L}$. The relation $(ab)^{-1}=b^{-1}a^{-1}$ holds in $L$, so \eqref{condinverse} is trivially satisfied for every pair of transversal lines, as well as for vertical lines. If $\ell_1$ and $\ell_2$ are neither both transversal nor both vertical, say $\ell_1=t_a$ and $\ell_2=v_b$, then 
\begin{equation*}
(t_av_b)^{-1}=v_{ab} \quad \text{and} \quad v_b^{-1}t_a^{-1}=v_{ba}.
\end{equation*}
Hence $(t_av_b)^{-1}=v_b^{-1}t_a^{-1}$ if, and only if, $L$ is commutative. The same condition is obtained by letting $(v_at_b)^{-1}=t_b^{-1}v_a^{-1}$. The equivalence between \eqref{propcondmoufang} and \eqref{propcondabel} follows from the fact that a commutative extra loop is an abelian group. From Remark~\ref{group}, we have the equivalence between \eqref{propcondassoc} and \eqref{propcondabel}.
\end{proof}

This completes the proof of Theorem~\ref{thm:Bolcond}.

\bigskip

\section{Nuclei and center of the extension} \label{sec:nuclei}

Our next aim is to describe the nuclei and center of $\widetilde{L}$.

\begin{proposition}\label{proprightnucleus}
Let $L$ be a right Bol loop with $x^2\in Z(L)$ for every $x\in L$. The right nucleus of $\widetilde{L}$  is
\begin{equation}
N_\rho(\widetilde{L})=\{t_z, \ v_z \mid z\in Z(L) \}.
\end{equation}
\end{proposition}

\begin{proof}
Consider $t_z\in N_\rho(\widetilde{L})$. Since $\mathcal{T}$ is isomorphic to $L$, the equation
\begin{equation*}
t_a(t_bt_z)=(t_at_b)t_z, \quad \forall \ a,b\in L,
\end{equation*}
is equivalent to saying that $z$ belongs to $N_\rho(L)$. One can easily check that requiring $t_a(v_bt_z)=(t_av_b)t_z$ leads to the same condition. Now, by setting
\begin{equation*}
v_a(v_bt_z)=(v_av_b)t_z, \quad \forall \ a,b\in L,
\end{equation*}
we obtain
\begin{equation}\label{condcomm}
a(bz^{-1})^{-1}=(ab^{-1})z, \quad \forall \ a,b\in L,
\end{equation}
which, since $z$ is a right nuclear element of $L$, is equivalent to $a(bz^{-1})^{-1}=a(b^{-1}z)$, which can be further simplified to
\begin{equation}\label{condbr}
(bz^{-1})^{-1}=b^{-1}z, \quad \forall \ b\in L.
\end{equation}
Since again $z$ belongs to $N_\rho(L)$, in $L$ the equation $(bz^{-1})^{-1}=zb^{-1}$ holds for every $b$, allowing us to rewrite \eqref{condbr} as 
\begin{equation*}
zb^{-1}=b^{-1}z, \quad \forall \ b\in L,
\end{equation*}
that is, $z$ belongs to the commutant $C(L)$. In conclusion, we have
\begin{equation*}
t_z\in N_\rho(\widetilde{L}) \iff z\in N_\rho(L)\cap C(L)=Z(L).
\end{equation*}

Let now $v_z\in N_\rho(\widetilde{L})$. Both the conditions 
\begin{equation*}
t_a(v_bv_z)=(t_av_b)v_z \quad \text{and} \quad t_a(t_bv_z) = (t_at_b) v_z, \quad \forall a,b\in L,
\end{equation*}
lead easily to $z\in N_\rho(L)$. The remaining conditions $v_a(v_bv_z)=(v_av_b)v_z$ and $v_a(t_bv_z)=(v_at_b)v_z$ are both equivalent to \eqref{condcomm}, leading to $z\in C(L)$. Therefore, we can say that also for vertical lines it holds
\begin{equation*}
v_z\in N_\rho(\widetilde{L}) \iff z\in N_\rho(L)\cap C(L)=Z(L). \qedhere
\end{equation*}
\end{proof}

While Proposition~\ref{proprightnucleus} provides a complete description of the right nucleus of the Bol loop $\widetilde{L}$, which remains the same for all suitable $L$, the situation is different for the left nucleus, as shown by the next result.

\begin{proposition}\label{propleftnucleus}
Let $L$ be a right Bol loop with $x^2\in Z(L)$ for every $x\in L$. For the left nucleus of $\widetilde{L}$ it holds that $$N_\lambda(\widetilde{L})\cap \mathcal{T}=\{t_n \mid n\in N_\lambda(L) \}\cong N_\lambda(L).$$ Furthermore,
\[N_\lambda(\widetilde{L})\cap \mathcal{V}=\{v_n \mid n=((na)b)(a^{-1}b^{-1}) \quad \forall a,b \in L\}\]
In particular:
\begin{enumerate}[(i)]
\item\label{propfirstcond} if $L$ is a non-abelian group, then
\begin{equation}
N_\lambda(\widetilde{L})=\mathcal{T}\cong L;
\end{equation}
\item\label{propseccond} if $L$ is an AIP loop, then
\begin{equation}
N_\lambda(\widetilde{L})=\{t_n, \ v_n \mid n\in N_\lambda(L) \}.
\end{equation}

\end{enumerate}
\end{proposition}

\begin{proof}
Since $\mathcal{T}$ is isomorphic to $L$, $$N_\lambda(\widetilde{L})\cap \mathcal{T}\subseteq \{t_n\mid n \in N_\lambda(L)\}\cong N_\lambda(L).$$
The conditions
\begin{equation*}
t_n(v_av_b)=(t_nv_a)v_b, \quad t_n(t_av_b)=(t_nt_a)v_b, \quad t_n(v_at_b)=(t_nv_a)t_b
\end{equation*}
are satisfied for any $a,b\in L$ and $n\in N_\lambda(L)$. Therefore, we can conclude that
\begin{equation*}
N_\lambda(\widetilde{L})\cap \mathcal{T}=\{t_n\mid n \in N_\lambda(L)\}.
\end{equation*}

Consider now $v_n\in N_\lambda(\widetilde{L})$. Equation $v_n(v_av_b)=(v_nv_a)v_b$ is equivalent to
\begin{equation}\label{condvertleft}
n(ab^{-1})^{-1}=(na^{-1})b \quad \forall \ a,b\in L.
\end{equation}

It is easy to see that the remaining associativity conditions 
\begin{equation*}
v_n(t_at_b)=(v_nt_a)t_b, \quad v_n(t_av_b)=(v_nt_a)v_b, \quad v_n(v_at_b)=(v_nv_a)t_b,
\end{equation*}
are all equivalent to \eqref{condvertleft}. Furthermore, \eqref{condvertleft} is equivalent to
\begin{equation}\label{condvertleft2}
n=((na)b)(a^{-1}b^{-1}) \quad \forall a,b\in L.
\end{equation}

If $L$ is a non-abelian group, then \eqref{condvertleft2} is impossible, therefore the left nucleus of $\widetilde{L}$ contains no vertical lines, that is,
\begin{equation*}
N_\lambda(\widetilde{L})\cap \mathcal{V}=\emptyset,
\end{equation*}
and \eqref{propfirstcond} is proved.
Lastly, if $L$ is an AIP loop, then \eqref{condvertleft2} requires $n$ to belong to the left nucleus $N_\lambda(L)$. Hence
\begin{equation*}
N_\lambda(\widetilde{L})\cap \mathcal{V}=\{v_n\mid \in N_\lambda(L)\},
\end{equation*}
which proves the assertion \eqref{propseccond}. 
\end{proof}

Propositions~\ref{proprightnucleus} and~\ref{propleftnucleus} imply Theorem~\ref{thm:nuclei}.

By Proposition~\ref{propleftnucleus} we can see that, differently from the right nucleus, the left nucleus has a less schematic description. We can say that, if $L$ is not AIP and $n$ belongs to $N_\lambda(L)$, then, from the equation \eqref{condvertleft}, $v_n$ cannot be an element of $N_\lambda(\widetilde{L})$. That is, the following holds.
\begin{equation}
N_\lambda(\widetilde{L})\cap \mathcal{V} \subseteq \left\{ v_n \mid n\in L\setminus N_\lambda(L)\right\}.
\end{equation}
Among the $2038$ right Bol loops of size $16$, $1940$ have central squares. We performed an analysis using the GAP package LOOPS \cite{loops3.4.3} and identified $1773$ of them that are neither associative nor AIP loops. Of these, there are $14$ cases where the left nucleus of the corresponding $\widetilde{L}$ has the largest intersection possible with $\mathcal{V}$, that is, $N_\lambda(\widetilde{L})\cap \mathcal{V} = \left\{ v_n \mid n\in L\setminus N_\lambda(L)\right\}$. For example $L=\mathrm{RightBolLoop}(16,181)$ has left nucleus equal to $\{1,3,4,7,9,11,12,15\}$ and $N_\lambda(\widetilde{L})\cap \mathcal{V}$ is $\{v_2, v_5, v_6, v_8, v_{10}, v_{13}, v_{14}, v_{16}\}$.

Let $L$ be a right Bol loop of order $16$ with central squares, assume that $L$ is not AIP, and $\nu(L):=|N_\lambda(L)\cap \mathcal{V}|$. Also, let us denote by $\mu_k$ the size of the family $\{L\mid \nu(L)=k\}$, $k=0,\dots,16$. With LOOPS \cite{loops3.4.3} we find the output listed in Table~\ref{tab:vertnucleus}.
\begin{table}[H]
\centering
\begin{tabular}{c|ccccccccccccccccc}

$k$ & 0 & 1 & 2 & 3 & 4 & 5 & 6 & 7 & 8 & 9 & 10 & 11 & 12 & 13 & 14 & 15 & 16  \\
\hline
$\mu_k$ & 1145 & 0 & 454 & 0 & 160 & 0 & 0 & 0 & 14 & 0 & 0 & 0 & 0 & 0 & 0 & 0 & 0
\end{tabular}
\caption{Number of loops with $\nu(L)=k$ }
\label{tab:vertnucleus}
\end{table}

\bigskip

\begin{proposition}\label{prop:center}
Let $L$ be a right Bol loop with $x^2\in Z(L)$ for every $x\in L$, and let  $Z(\widetilde{L})$ be the center of $\widetilde{L}$. The following assertions hold:
\begin{enumerate}[(i)]
\item if $L$ has not exponent $2$, then
\begin{equation}
Z(\widetilde{L})=\{t_z \mid z\in Z(L) \;\text{and}\; z^2=1 \};
\end{equation}
\item if $L$ has exponent $2$, then
\begin{equation}
Z(\widetilde{L})=\{t_z, \ v_z \mid z\in Z(L) \}.
\end{equation}
\end{enumerate}
\end{proposition}

\begin{proof}
By Proposition~\ref{proprightnucleus}, $Z(\widetilde{L}) \subseteq \{t_z,v_z \mid z\in Z(L)\}$. If $L$ has no exponent $2$, no vertical line can be in the commutant, since the condition $v_zt_a=t_av_z$, for any $a\in L$, is equivalent to $za^{-1}=az$, that is $a=a^{-1}$ for any $a\in L$. Hence the center of $\widetilde{L}$ consists only of transversal lines $t_z$ for some $z\in Z(L)$. The condition $t_zt_a=t_at_z$ is naturally satisfied, and $t_zv_a=v_at_z$ requires that $z^2=1$. 	Hence
\begin{equation*}
Z(\widetilde{L})=\{t_z \mid z\in Z(L) \;\text{and}\; z^2=1 \}.
\end{equation*}
If instead $L$ has exponent $2$, then
\begin{equation*}
N(\widetilde{L})=\{t_z, \ v_z \mid z\in Z(L)\},
\end{equation*}
and with the same arguments we obtain that $Z(L)=N(L)$.
\end{proof}

Propositions~\ref{proprightnucleus},~\ref{propleftnucleus} and~\ref{prop:center} prove Theorem~\ref{thm:nuclei}.

Now we determine when this process can be iterated to produce extensions of higher indices. To do this, we need to find when the extension $\widetilde{L}$ of index two of $L$ is again a Bol loop with central squares. Of course, if $L$ has exponent two this process can be naturally iterated, since in this case $\widetilde{L}$ is $L\times C_2$.

\begin{proposition}
   The right Bol loop $\widetilde{L}$ has central squares if and only if the exponent of $L$ divides $4$.
\end{proposition}

\begin{proof}
    Let $L$ be a right Bol loop with central squares. As we noticed before, $v_a^2=t_1$ for every $a\in L$. Therefore, we only need to check when the squares of the transversal lines are central. If $L$ has exponent $2$, then this condition holds trivially. Let $\mathrm{exp}(L)\neq 2$. By Proposition \ref{prop:center}, we have that the element $t_a^2=t_{a^2}$ is central if and only if $a^4=1$. 
\end{proof}

\bigskip

\section{The core of the extension of a right Bol loop} \label{sec:extcore}

Now we want to see how to define a \emph{quandle} starting from a Bol loop. A quandle, which is a special case of a rack, is a set with a binary operation satisfying axioms reflecting the three Reidemeister moves (c.f. \cite{Knotentheorie}) in knot theory. Although mainly used to study invariants for knots, quandles hold intrinsic interest as algebraic structures. In particular, the definition of a quandle axiomatizes the properties of conjugation within a group. 

\begin{definition}
A \emph{quandle} $(Q,\triangleleft)$ is a set $Q$ with a binary operation $\triangleleft \colon Q\times Q \longrightarrow Q$ which satisfies the following axioms:
\begin{enumerate}
\item for all $a\in Q$, $a\triangleleft a=a$ (idempotence);
\item for all $a,b\in Q$, there exists a unique $x\in Q$ such that $x\triangleleft a=b$;
\item for all $a,b,c\in Q$, $(a\triangleleft b)\triangleleft c= (a\triangleleft c)\triangleleft (b\triangleleft c)$ (right self-distributivity).
\end{enumerate}
\end{definition}
In other words, a quandle can be described as a right quasigroup $Q$ where the right translations are automorphism of $Q$, with the element $a$ itself being fixed by $R_a$.

\begin{example}
If $(G,\cdot)$ is a group, the conjugation $a\triangleleft b:= b^{-1}ab$ gives a quandle structure. 
\end{example}

A quandle $Q$ is said \emph{involutory} if $(a\triangleleft b)\triangleleft b=a$ for every $a,b\in Q$, and it is said \emph{connected} if the right multiplication group 
\begin{equation*}
\mathrm{RMlt}(Q)=\langle R_b\mid b\in Q, \ a^{R_b}=a\triangleleft b \rangle
\end{equation*}
acts transitively on $Q$. Connected quandles are the main objects in \cite{connectedquandles}, where the authors present a correspondence between with certain configurations in transitive groups, called quandle envelopes. 

If $(L,\cdot)$ is a loop, the \emph{core} $(L,+)$ of $L$ is defined by the operation
\begin{equation}
x+y:=(yx^{-1}) y.
\end{equation}
If $L$ is a right Bol loop, its core is an involutory quandle (see \cite{stanovsky}).

If $L$ is a right Bol loop with central squares, we define the core of $\widetilde{L}$ in the same way, denoting it with the same symbol $+$.

\begin{proof}[Proof of Theorem~\ref{thm:core}]
Since the subloop $\mathcal{T}$ of $\widetilde{L}$ is $L$, $t_a+t_b=t_{a+b}$. Moreover,
\begin{equation*}
v_a+v_b=v_bv_a\cdot v_b= t_{ba^{-1}}v_b=v_{ba^{-1}\cdot b}=v_{a+b}.
\end{equation*}
Hence, the core of $\widetilde{L}$ decomposes into the disjoint union of two cores $\mathcal{T}$ and $\mathcal{V}$ both isomorphic to the core of $L$. 
\end{proof}

Furthermore, for the mixed computations it holds
\begin{align} \label{eq:tildecore}
t_a+v_b&=t_{ba\cdot b^{-1}}, & v_a+t_b&=v_{ba\cdot b^{-1}}.
\end{align}

We note here that the equation $t_a+\ell=t_{ba\cdot b^{-1}}$ has $v_b$ and $v_{b^{-1}}$ as solutions, so in general the quandle $(\widetilde{L},+)$ is not a quasigroup.

\begin{definition}
The \emph{structure group} of a quandle $(Q,\triangleleft)$ is
\begin{equation*}
\STR(Q,\triangleleft)=\langle g_a, \ a\in Q \mid g_a g_b= g_b g_{a\triangleleft b}, \ a,b\in Q \rangle.
\end{equation*}
\end{definition}

The structure group of a finite quandle is either free abelian of rank $r$, where $r$ is the number of orbits of $Q$ with respect to all the right translations, or non-abelian and with torsion. In the latter case, $\STR(Q,\triangleleft)$ has a finite index free abelian subgroup of rank $r$; see \cite{abelianquandles} for more details.

If $L$ is a Bol loop, then the idempotency of the core implies 
\[g_a=g_{(a+b)+b}=g_b^{-1} g_{a+b} g_b = g_b^{-2}g_ag_b^2.\]
Hence, $g_b^2 \in Z(\STR(L,+))$ for all $b\in L$. Moreover,
\[g_{a+b}^2=(g_b^{-1}g_ag_b)^2=g_b^{-1}g_a^2g_b = g_a^2.\]
Let $a_1,\ldots,a_r$ be orbit representatives of the right translation group of the core. The subgroup $T=\langle g_{a_1}^2,\ldots,g_{a_r}^2 \rangle$ is a free abelian of rank $r$, which is contained in the center of the structure group. As shown above, $T$ contains every $g_a^2$, $a\in L$. We call the factor $\STR(L,+)/T$ the \textit{restricted structure group}
\begin{equation*}
\rSTR(L,+)=\langle \hat{g}_a, \ a\in L \mid  \hat{g}_b\hat{g}_a \hat{g}_b=\hat{g}_{a+ b}, \ \hat{g}_a^2=1\ a,b\in L \rangle.
\end{equation*}
The Bol reflections $\sigma_a$ satisfy $\sigma_a^2=\mathrm{Id}$ and $\sigma_{a+b}=\sigma_b\sigma_a\sigma_b$. Thus, $\hat{g}_a\longmapsto \sigma_a$ extends to a surjective homomorphism from the restricted structure group onto the collineation group 
\[\Gamma=\langle \sigma_a \mid a\in L \rangle\] 
generated by Bol reflections. 

\begin{lemma} \label{lm:rankGG'}
Let $L$ be a right Bol loop, and denote by $r$ the number of orbits of the right multiplication group of the core. Write $G=\rSTR(L,+)$. Then $G/G'\cong C_2^r$.  
\end{lemma}
\begin{proof}
Let $a_1,\ldots,a_r$ be orbit representatives and let $h_1,\ldots,h_r$ be the free generators of the elementary abelian $2$-group $C_2^r$. For an element $a$ in the orbit of $a_i$, we define the map $\Phi:\hat{g}_a \to h_i$. $\Phi$ extends to a surjective homomorphism $G\to C_2^r$. Hence, $G/G'$ is elementary abelian of rank at least $r$. If $a,b\in L$ are in the same orbit, then $\hat{g}_a^t=\hat{g}_b$ for some $t\in \rSTR(L,+)$. This implies $\hat{g}_a\hat{g}_b \in G'$, or equivalently $\hat{g}_aG'=\hat{g}_b G'$. Therefore, the rank of $G/G'$ cannot be more than $r$. 
\end{proof}

We finish this section by computing the restricted structure group of the cores of right Bol loops in some cases. 

\begin{proposition} \label{prop:rSRTdirprod}
Let $L$ be a right Bol loop, and denote by $r$ the number of orbits of the right multiplication group of the core. Then
\[\rSTR(L\times C_2,+) \cong \rSTR(L,+) \times C_2^r.\]
\end{proposition}
\begin{proof}
We denote the generators of $\rSTR(L\times C_2,+)$ by $\hat{g}_{a,i}$ with $a\in L$, $i\in \mathbb{F}_2$. $\hat{g}_{b,j}\hat{g}_{a,i}\hat{g}_{b,j}=\hat{g}_{a+b,i}$ implies that $\hat{g}_{b,0},\hat{g}_{b,1}$ have the same action on all generators. This means that $c_b=\hat{g}_{b,0}\hat{g}_{b,1}$ is in the center, and
\[\rSTR(L\times C_2,+) \cong \rSTR(L,+) \times \langle c_b \mid b \in L \rangle.\]
On the one hand, $c_b^2=1$. On the other hand,
\[c_b=\hat{g}_{a,0}c_b\hat{g}_{a,0} = \hat{g}_{a,0}\hat{g}_{b,0}\hat{g}_{b,1}\hat{g}_{a,0} = \hat{g}_{b+a,0}\hat{g}_{b+a,1} = c_{b+a}.\]
These show that $\langle c_b \mid b \in L \rangle$ is an elementary abelian $2$-group 
of rank at most $r$. Lemma~\ref{lm:rankGG'} implies that the rank is equal to $r$. 
\end{proof}

As for a right Bol loop $L$ of exponent $2$, the extension $\widetilde{L}$ is simply the direct product $L\times C_2$, the following result is immediate.

\begin{corollary} \label{coro:exp2core}
Let $L$ be a finite right Bol loop of exponent $2$. Then
\[\rSTR(\widetilde{L},+)\cong \rSTR(L,+)  \times C_2^r, \]
where $r$ is the number of orbits of the right multiplication group of the core.
\end{corollary}

The restricted structure group can also be computed for another class, which is rather trivial from the point of view of the theory of loops, namely for the class of abelian groups. However, in this case, the formula for the restricted structure group is surprisingly different.

\begin{proposition}
Let $L$ be an abelian group. Then 
\[\rSTR(\widetilde{L},+)\cong \rSTR(L,+)  \times \rSTR(L,+). \]
\end{proposition}
\begin{proof}
Define the subgroups
\begin{align*}
G_T&=\langle \hat{g}_{t_a} \mid a \in L \rangle, \\
G_V&=\langle \hat{g}_{v_a} \mid a \in L \rangle
\end{align*}
of $\rSTR(\widetilde{L},+)$. Clearly $\rSTR(\widetilde{L},+) = \langle G_T,G_V \rangle$, and $G_T\cong G_V \cong \rSTR(L,+)$. \eqref{eq:tildecore} implies that $\hat{g}_{v_b}$ commutes with $G_T$, and $\hat{g}_{t_b}$ commutes with $G_V$. The claim $\rSTR(\widetilde{L},+) \cong G_T\times G_V$ follows.
\end{proof}

\section{Acknowledgement}
This research was supported by project TKP2021-NVA-09, implemented with the support provided by the Ministry of Innovation and Technology of Hungary from the National Research, Development and Innovation Fund, financed under the TKP2021-NVA funding scheme. Partially supported by the NKFIH-OTKA Grant SNN 132625.

\bigskip

\bigskip

\noindent \textbf{Conflict of Interest Statement:}\\
The authors declare that there are no financial or non-financial conflicts of interest related to the publication of this article.

\printbibliography

\end{document}